\documentclass[a4paper,10pt]{amsart}
\usepackage{amsmath,amssymb,amsthm}
\title{Note on a paper ``An Extension of a Theorem of Euler'' by Hirata-Kohno et al.}
\author{Sz. Tengely}
\thanks{Research supported in part by the Magyary Zolt\'an Higher Educational Public Foundation}
\address{Mathematical Institute\newline
 \indent University of Debrecen\newline
 \indent P.O.Box 12\newline
 \indent 4010 Debrecen\newline
 \indent Hungary}
\email{tengely@math.klte.hu}

\begin{document}
\newtheorem{thm}{Theorem}
\newtheorem{lem}{Lemma}
\newtheorem*{cor}{Corollary}
\newtheorem*{thm1}{Theorem}
\newtheorem*{thmA}{Theorem A}
\newtheorem*{thmB}{Theorem B}
\newtheorem*{lem1}{Lemma}
\newtheorem*{conj1}{Conjecture}
\theoremstyle{definition}
\newtheorem*{rem}{Remark}
\newtheorem*{acknowledgement}{Acknowledgement}
\bibliographystyle{plain}
\maketitle

\begin{abstract}
In this paper we extend a result of Hirata-Kohno, Laishram, Shorey and Tijdeman on the Diophantine equation 
$n(n+d)\cdots(n+(k-1)d)=by^2,$ where $n,d,k\geq 2$ and $y$ are positive integers such that $\gcd(n,d)=1.$
\end{abstract}

\section{introduction}
Let $n,d,k>2$ and $y$ be positive integers such that $\gcd(n,d) = 1.$  For an integer $\nu> 1,$ we denote by $P(\nu)$ the greatest prime factor of $\nu$ and we put $P(1) = 1.$ Let $b$ be a squarefree positive integer such that $P(b)\leq k.$ We consider the equation
\begin{equation}\label{1}
n(n + d)\cdots(n + (k-1)d) = by^2
\end{equation}
in $n,d,k$ and $y.$ 

A celebrated theorem of Erd\H{o}s and Selfridge \cite{ES} states that the product of consecutive positive integers is never a perfect power. An old, difficult conjecture states that even a product of consecutive terms of arithmetic progression of length $k>3$ and difference $d\geq 1$ is never a perfect power.
Euler proved (see \cite{Dickson} pp. 440 and 635) that a product of four terms in arithmetic progression is never a square solving equation \eqref{1} with $b=1$ and $k=4.$ Obl\'ath \cite{Oblath} obtained a similar statement for $b=1,k=5.$ Bennett, Bruin, Gy\H{o}ry and Hajdu \cite{BBGyH} solved \eqref{1} with $b=1$ and $6\leq k\leq 11.$  For more results on this topic see \cite{BBGyH}, \cite{HLST} and the references given there.

We write
\begin{equation}\label{2}
n + id = a_ix_i^2 \mbox{ for } 0\leq i<k
\end{equation}
where $a_i$ are squarefree integers such that $P(a_i)\leq \max(P(b),k-1)$ and $x_i$ are positive
integers. Every solution to \eqref{1} yields a $k$-tuple $(a_0, a_1,\ldots, a_{k-1}).$ Recently Hirata-Kohno, Laishram, Shorey and Tijdeman \cite{HLST} proved the following theorem.

\begin{thmA}[Hirata-Kohno, Laishram, Shorey,Tijdeman]
Equation \eqref{1} with $d > 1,P(b) = k$ and $7\leq k\leq 100$ implies that $(a_0,a_1,\ldots,a_{k-1})$
is among the following tuples or their mirror images.
\begin{eqnarray*}
k = 7 : && (2,3,1,5,6,7,2), (3, 1, 5, 6, 7, 2, 1), (1, 5, 6, 7, 2, 1, 10),\\
k = 13 : &&  (3, 1, 5, 6, 7, 2, 1, 10, 11, 3, 13, 14, 15),\\
&& (1, 5, 6, 7, 2, 1, 10, 11, 3, 13, 14, 15, 1),\\
k = 19 : && (1, 5, 6, 7, 2, 1, 10, 11, 3, 13, 14, 15, 1, 17, 2, 19, 5, 21, 22),\\
k = 23 : && (5, 6, 7, 2, 1, 10, 11, 3, 13, 14, 15, 1, 17, 2, 19, 5, 21, 22, 23, 6, 1, 26, 3),\\
&& (6, 7, 2, 1, 10, 11, 3, 13, 14, 15, 1, 17, 2, 19, 5, 21, 22, 23, 6, 1, 26, 3, 7).
\end{eqnarray*}
\end{thmA}

In case of $k=5$ Bennett, Bruin, Gy\H{o}ry  and Hajdu \cite{BBGyH} proved the following result.
\begin{thmB}[Bennett, Bruin, Gy\H{o}ry, Hajdu]
If $n$ and $d$ are coprime nonzero integers, then the Diophantine equation
$$
n(n+d)(n+2d)(n+3d)(n+4d)=by^2
$$
has no solutions in nonzero integers $b,y$ and $P(b)\leq 3.$
\end{thmB}

In this article we solve \eqref{1} with $k=5$ and $P(b)=5,$ moreover we handle the 8 special cases mentioned in Theorem A. We prove the following theorems.

\begin{thm}\label{k7}
Equation \eqref{1} with $d > 1,P(b) = k$ and $7\leq k\leq 100$ has no solutions.
\end{thm}

\begin{thm}\label{k5}
Equation \eqref{1} with $d>1,k=5$ and $P(b)=5$ implies that $(n,d)\in\{(-12,7),(-4,3)\}.$
\end{thm}

\section{preliminary lemmas}
In the proofs of Theorem \ref{k5} and \ref{k7} we need several results using elliptic Chabauty's method (see \cite{NB1},\cite{NB2}). Bruin's routines related to elliptic Chabauty's method are contained in MAGMA \cite{MAGMA} so here we give the appropriate computations only.

\begin{lem}\label{15672110}
Equation \eqref{1} with $k=7$ and $(a_0,a_1,\ldots,a_6)=(1,5,6,7,2,1,10)$ implies that $n=2,d=1.$
\end{lem}
\begin{proof}
Using that $n=x_0^2$ and $d=(x_5^2-x_0^2)/5$ we obtain the following system of equations
\begin{eqnarray*}
&&x_5^2+4x_0^2=25x_1^2,\\
&&4x_5^2+x_0^2=10x_4^2,\\
&&6x_5^2-x_0^2=50x_6^2.
\end{eqnarray*}
The second equation implies that $x_0$ is even, that is there exists a $z\in\mathbb{Z}$ such that $x_0=2z.$
By standard factorization argument in the Gaussian integers we get that
$$
(x_5+4iz)(x_5+iz)=\delta\square,
$$
where $\delta\in\{-3\pm i, -1\pm 3i, 1\pm 3i, 3\pm i\}.$ Thus putting $X=x_5/z$ it is sufficient to find all points $(X,Y)$ on the curves 
\begin{equation}\label{C1}
C_{\delta}:\quad \delta(X+i)(X+4i)(3X^2-2)=Y^2,
\end{equation}
where $\delta\in\{-3\pm i, -1\pm 3i, 1\pm 3i, 3\pm i\},$ for which $X\in\mathbb{Q}$ and $Y\in\mathbb{Q}(i).$ Note that if $(X,Y)$ is a point on $C_{\delta}$ then $(X,iY)$ is a point on $C_{-\delta}.$ We will use this isomorphism later on to reduce the number of curves to be examined. Hence we need to consider the curve $C_{\delta}$ for $\delta\in\{1-3i,1+3i,3-i,3+i\}.$ 

I. $\delta=1-3i.$ In this case $C_{1-3i}$ is isomorphic to the elliptic curve
$$
E_{1-3i}:\quad y^2 = x^3 + ix^2 + (-17i - 23)x +(2291i + 1597).
$$
Using MAGMA we get that the rank of $E_{1-3i}$ is 0 and there is no point on $C_{1-3i}$ for which $X\in\mathbb{Q}.$

II. $\delta=1+3i.$ Here we obtain that $E_{1+3i}: y^2 = x^3 - ix^2 + (17i - 23)x +(-2291i + 1597).$ The rank of this curve is 0 and there is no point on $C_{1+3i}$ for which $X\in\mathbb{Q}.$

III. $\delta=3-i.$ The elliptic curve in this case is $E_{3-i}: y^2 = x^3 + x^2 + (-17i + 23)x +
    (-1597i - 2291).$ We have $E_{3-i}(\mathbb{Q}(i))\simeq \mathbb{Z}_2\oplus\mathbb{Z}$ as an Abelian group. Applying elliptic Chabauty with $p=13,$ we get that $x_5/z=-3.$ Thus $n=2$ and $d=1.$

IV. $\delta=3+i.$ The curve $C_{3+i}$ is isomorphic to $E_{3+i}: y^2 = x^3 + x^2 + (17i + 23)x +
    (1597i - 2291).$ The rank of this curve is 1 and applying elliptic Chabauty again with $p=13$ we obtain that $x_5/z=3.$ This implies that $n=2$ and $d=1.$ 
\end{proof}

\begin{lem}\label{2315672}
Equation \eqref{1} with $k=7$ and $(a_0,a_1,\ldots,a_6)=(2,3,1,5,6,7,2)$ implies that $n=2,d=1.$
\end{lem}
\begin{proof}
In this case we have the following system of equations
\begin{eqnarray*}
&&x_4^2+x_0^2=2x_1^2,\\
&&9x_4^2+x_0^2=10x_3^2,\\
&&9x_4^2-x_0^2=2x_6^2.
\end{eqnarray*}
Using the same argument as in the proof of Theorem 1 it follows that it is sufficient to find all points $(X,Y)$ on the curves 
\begin{equation}\label{C2}
C_{\delta}:\quad 2\delta(X+i)(3X+i)(9X^2-1)=Y^2,
\end{equation}
where $\delta\in\{-4\pm 2i, -2\pm 4i, 2\pm 4i, 4\pm 2i\},$ for which $X\in\mathbb{Q}$ and $Y\in\mathbb{Q}(i).$
We summarize the results obtained by elliptic Chabauty in the following table. In each case we used $p=29.$
\begin{center}
\begin{tabular}{|c|c|c|}
\hline
$\delta$ & curve & $x_4/x_0$ \\ \hline
$2-4i$ & $y^2 = x^3 + (-12i - 9)x + (-572i - 104)$ &  $\{-1,\pm 1/3\}$\\ \hline
$2+4i$ & $y^2 = x^3 + (12i - 9)x + (-572i + 104)$ &  $\{1,\pm 1/3\}$\\ \hline
$4-2i$ &  $y^2 = x^3 + (-12i + 9)x + (-104i - 572)$ & $\{\pm 1/3\}$ \\ \hline
$4+2i$& $y^2 = x^3 + (12i + 9)x + (-104i + 572)$ & $\{\pm 1/3\}$\\
\hline
\end{tabular}
\end{center}
Thus $x_4/x_0\in\{\pm 1,\pm 1/3\}.$ From $x_4/x_0=\pm 1$ it follows that $n=2,d=1,$ while $x_4/x_0=\pm 1/3$ does not yield any solutions.
\end{proof}

\begin{lem}\label{3156721}
Equation \eqref{1} with $k=7$ and $(a_0,a_1,\ldots,a_6)=(3,1,5,6,7,2,1)$ implies that $n=3,d=1.$
\end{lem}
\begin{proof}
Here we get the following system of equations
\begin{eqnarray*}
&&2x_3^2+2x_0^2=x_1^2,\\
&&4x_3^2+x_0^2=5x_2^2,\\
&&12x_3^2-3x_0^2=x_6^2.
\end{eqnarray*}
Using the same argument as in the proof of Theorem 1 it follows that it is sufficient to find all points $(X,Y)$ on the curves 
\begin{equation}\label{C3}
C_{\delta}:\quad \delta(X+i)(2X+i)(12X^2-3)=Y^2,
\end{equation}
where  $\delta\in\{-3\pm i, -1\pm 3i, 1\pm 3i, 3\pm i\}$ for which $X\in\mathbb{Q}$ and $Y\in\mathbb{Q}(i).$
We summarize the results obtained by elliptic Chabauty in the following table. In each case we used $p=13.$
\begin{center}
\begin{tabular}{|c|c|c|}
\hline
$\delta$ & curve & $x_3/x_0$ \\ \hline
$1-3i$ & $y^2 = x^3 + (27i + 36)x + (243i - 351)$ &  $\{-1,\pm 1/2\}$\\ \hline
$1+3i$ & $y^2 = x^3 + (-27i + 36)x + (243i + 351)$ &  $\{1,\pm 1/2\}$\\ \hline
$3-i$ &  $ y^2 = x^3 + (27i - 36)x + (-351i + 243)$ & $\{\pm 1/2\}$ \\ \hline
$3+i$& $y^2 = x^3 + (-27i - 36)x + (-351i - 243) $ & $\{\pm 1/2\}$\\
\hline
\end{tabular}
\end{center}
Thus $x_3/x_0\in\{\pm 1,\pm 1/2\}.$ From $x_4/x_0=\pm 1$ it follows that $n=3,d=1,$ while $x_3/x_0=\pm 1/2$ does not yield any solutions.
\end{proof}

\begin{lem}\label{35211}
Equation \eqref{1} with $k=5,d>1$ and $(a_0,a_1,\ldots,a_4)=(-3,-5,2,1,1)$ implies that $n=-12,d=7.$
\end{lem}
\begin{proof}
From the system of equations \eqref{2} we have
\begin{eqnarray*}
&&\frac{1}{4}x_4^2-\frac{9}{4}x_0^2=-5x_1^2,\\
&&\frac{1}{2}x_4^2-\frac{3}{2}x_0^2=2x_2^2,\\
&&\frac{3}{4}x_4^2-\frac{3}{4}x_0^2=x_3^2.
\end{eqnarray*}
Clearly, $\gcd(x_4,x_0)=1$ or 2. In both cases we get the following system of equations
\begin{eqnarray*}
&&X_4^2-9X_0^2=-5\square,\\
&&X_4^2-3X_0^2=\square,\\
&&X_4^2-X_0^2=3\square,
\end{eqnarray*}
where $X_4=x_4/\gcd(x_4,x_0)$ and $X_0=x_0/\gcd(x_4,x_0).$
The curve in this case is
$$
C_{\delta}:\quad \delta(X+\sqrt{3})(X+3)(X^2-1)=Y^2,
$$
where $\delta$ is from a finite set. Elliptic Chabauty's method applied with $p=11,37$ and 59 provides all points for which the first coordinate is rational. These coordinates are $\{-3,-2,-1,1,2\}.$ We obtain the arithmetic progression with $(n,d)=(-12,7).$
\end{proof}

\begin{lem}\label{25211}
Equation \eqref{1} with $k=5,d>1$ and $(a_0,a_1,\ldots,a_4)=(2,5,2,-1,-1)$ implies that $n=-4,d=3.$
\end{lem}
\begin{proof}
We use $x_3$ and $x_2$ to get a system of equations as in the previous lemmas. Elliptic Chabauty's method applied with $p=13$ yields that $x_3/x_2=\pm 1,$ hence $(n,d)=(-4,3).$
\end{proof}

\begin{lem}\label{65132}
Equation \eqref{1} with $k=5,d>1$ and $(a_0,a_1,\ldots,a_4)=(6,5,1,3,2)$ has no solutions.
\end{lem}
\begin{proof}
In this case we have
$$
\delta(x_3+\sqrt{-1}x_0)(x_3+2\sqrt{-1}x_0)(2x_3^2-x_0^2)=\square,
$$
where $\delta\in\{1\pm 3\sqrt{-1},3\pm\sqrt{-1}\}.$ Chabauty's argument gives $x_3/x_0=\pm 1,$ which corresponds to arithmetic progressions with $d=\pm 1.$
\end{proof}

\section{remaining cases of Theorem A}
In this section we prove Theorem \ref{k7}. 
\begin{proof}
First note that Lemmas \ref{15672110}, \ref{2315672} and \ref{3156721} imply the statement of the theorem in cases of $k=7,13$ and 19. The two remaining possibilities can be eliminated in a similar way, we present the argument working for the tuple
$$
(5, 6, 7, 2, 1, 10, 11, 3, 13, 14, 15, 1, 17, 2, 19, 5, 21, 22, 23, 6, 1, 26, 3).
$$
We have the system of equations
\begin{eqnarray*}
&&n+d=6x_1^2,\\
&&n+3d=2x_3^2,\\
&&n+5d=10x_5^2,\\
&&n+7d=3x_7^2,\\
&&n+9d=14x_9^2,\\
&&n+11d=x_{11}^2,\\
&&n+13d=2x_{13}^2.
\end{eqnarray*}
We find that $x_7,x_{11}$ and $(n+d)$ are even integers. Dividing all equations by 2 we obtain an arithmetic progression of length 7 and $(a_0,a_1,\ldots,a_6)=(3,1,5,6,7,2,1).$ This is not possible by Lemma \ref{3156721} and the theorem is proved. 
\end{proof}

\section{the case $k=5$}
In this section we prove Theorem \ref{k5}.
\begin{proof}
Five divides one of the terms and by symmetry we may assume that $5\mid n+d$ or $5\mid n+2d.$ First we compute the set of possible tuples $(a_0,a_1,a_2,a_3,a_4)$ for which appropriate congruence conditions hold ($\gcd(a_i,a_j)\in\{1,P(j-i)\}$ for $0\leq i<j\leq 4$) and the number of sign changes are at most 1 and the product $a_0a_1a_2a_3a_4$ is positive. After that we eliminate tuples by using elliptic curves of rank 0. We consider elliptic curves $(n+\alpha_1d)(n+\alpha_2d)(n+\alpha_3d)(n+\alpha_4d)=\prod_ia_{\alpha_i}\square,$ where $\alpha_i,i\in\{1,2,3,4\}$ are distinct integers belonging to the set $\{0,1,2,3,4\}.$ If the rank is 0, then we obtain all possible values of $n/d.$ Since $\gcd(n,d)=1$ we get all possible values of $n$ and $d.$ It turns out that it remains to deal with the following tuples
\begin{eqnarray*}
&&(-3,-5,2,1,1),\\
&&(-2,-5,3,1,1),\\
&&(-1,-15,-1,-2,3),\\
&&(2,5,2,-1,-1),\\
&&(6,5,1,3,2).
\end{eqnarray*}
In case of $(-3,-5,2,1,1)$ Lemma \ref{35211} implies that $(n,d)=(-12,7).$

If $(a_0,a_1,\ldots,a_4)=(-2,-5,3,1,1),$ then by $\gcd(n,d)=1$ we have that $\gcd(n,3)=1.$ Since $n=-2x_0^2$ we obtain that $n\equiv 1\pmod{3}.$ From the equation $n+2d=3x_2^2$ we get that $d\equiv 1\pmod{3}.$ Finally, the equation $n+4d=x_4^2$ leads to a contradiction.

If $(a_0,a_1,\ldots,a_4)=(-1,-15,-1,-2,3),$ then we obtain that $\gcd(n,3)=1.$ From the equations $n=-x_0^2$ and $n+d=-15x_1^2$ we get that $n\equiv 2\pmod{3}$ and $d\equiv 1\pmod{3}.$ Now the contradiction follows from the equation $n+2d=-x_2^2.$

In case of the tuple $(2,5,2,-1,-1)$ Lemma \ref{25211} implies that $(n,d)=(-4,3).$ The last tuple is eliminated by Lemma \ref{65132}.
\end{proof}

\bibliography{all,BruinN}

\end{document}